\documentclass{amsart}
\usepackage[utf8x]{inputenc}
\usepackage[T1]{fontenc}
\usepackage{eucal}
\usepackage{lmodern}
\usepackage{amssymb}
\usepackage{amsmath}
\usepackage{amsthm}
\usepackage{cite}
\newtheorem{prop}{Proposition}
\title{Random Thue and Fermat equations}
\date{}
\author{Rainer Dietmann}
\address{Department of Mathematics, Royal Holloway, University of London,
Egham TW20 0EX, UK}
\email{Rainer.Dietmann@rhul.ac.uk}
\author{Oscar Marmon}
\address{Mathematisches Institut \\ Georg-August-Universität Göttingen \\ Bunsenstr.~ 3-5 \\ 37073 Göttingen \\ Germany}
\email{omarmon@uni-math.gwdg.de}

\begin{document}

\renewcommand{\bf}{\mathbf}
\newtheorem{thm}{Theorem}
\newtheorem{lemma}{Lemma}
\theoremstyle{remark}
\newtheorem*{rem*}{Remark}
\newcommand{\epsi}{\varepsilon}
\newcommand{\ZZ}{\mathbb{Z}}
\newcommand{\RR}{\mathbb{R}}
\newcommand{\NN}{\mathbb{N}}
\newcommand{\CC}{\mathbb{C}}
\newcommand{\xx}{\mathbf{x}}
\newcommand{\yy}{\mathbf{y}}
\newcommand{\cN}{\mathcal{N}}

\newcommand{\en}{\mathbb{N}}
\newcommand{\qu}{\mathbb{Q}}
\newcommand{\zet}{\mathbb{Z}}
\newcommand{\ce}{\mathbb{C}}
\newcommand{\er}{\mathbb{R}}
\newtheorem{theorem}{Theorem}
\newtheorem{corollary}{Corollary}

\begin{abstract}
We consider Thue equations of the form $ax^k+by^k = 1$, and
assuming the truth of the $abc$-conjecture, we show that almost all locally
soluble Thue equations of degree at least three violate the Hasse
principle. A similar conclusion holds true for Fermat equations
$ax^k+by^k+cz^k = 0$ of degree at least six.
\end{abstract}

\maketitle
\section{Introduction}
Let $F(X_1, \ldots, X_s) \in \zet[X_1, \ldots, X_s]$.
If $F(x_1, \ldots, x_s)=0$ for some $\mathbf{x} \in \zet^s$, then
trivially the equation $F(x_1, \ldots, x_s)=0$ also has solutions
over $\er$ and over all local rings $\zet_p$. If the opposite is
true as well, then we say that $F$ satisfies the \emph{Hasse principle}.
For homogeneous polynomials $F$, always $F(0, \ldots, 0)=0$, so one then
naturally asks for non-trivial solutions. The Hasse principle for
example holds true for quadratic forms, but fails for cubic forms:
one famous counterexample (see \cite{Selmer}) is given
by the cubic form
\begin{equation}
\label{cub345}
  F(X_1, X_2, X_3) = 3X_1^3 + 4X_2^3 + 5X_3^3.
\end{equation}
In recent years questions about the frequency of such failures of the
Hasse principle were addressed for different classes of Diophantine
equations (see for example \cite{Bhargava_1308.0395},
\cite{Bhargava_1402.1131}, \cite{Browning-Dietmann09},
\cite{Bruedern_Dietmann_14},
\cite{BDLB_14}, \cite{Poonen_Voloch}).
For hyperelliptic curves, Bhargava \cite{Bhargava_1308.0395} has
recently shown that asymptotically, as their genus tends to infinity,
their probability to satisfy the Hasse principle, given that there
are local solutions, tends to zero.
In this note we focus on curves as well, namely those that are given
by Fermat equations such as \eqref{cub345}, or given by Thue equations.
This way we provide families of curves which satisfy the Hasse principle
with probability zero, and already for fixed small degree rather than
asymptotically for the degree tending to infinity,
but on the other hand our results,
like those in a related earlier paper \cite{Granville},
are conditional on the
\emph{$abc$-conjecture} (see \cite{Masser}), which we briefly recall:
if $a+b=c$ with $a,b,c \in \zet$ where $(a,b,c)=1$, and
\[
  P = \prod_{p \, | \, abc} p,
\]
the product taken over all primes $p$ dividing $abc$, then
for all $\varepsilon>0$ we have
\[
  \max\{|a|, |b|, |c|\} \ll_\varepsilon P^{1+\varepsilon}.
\]
Assuming the $abc$-conjecture, we are able to show that a `random'
Thue equation of degree at least three
has an integer solution with probability $0$, even if
it is locally soluble everywhere.
To be more precise, let
\begin{align*}
  N_{k, \operatorname{loc}}(H)  = \# & \{a, b \in \zet: 0<|a|,|b| \le H \text{ and }
  ax^k+by^k=1\\
  & \text{has solutions over all local rings $\zet_p$ and over $\er$}\}\\
\end{align*}
and
\begin{align*}
  N_{k, \operatorname{glob}}(H) = \# & \{a, b \in \zet: 0<|a|,|b| \le H \text{ and }
  ax^k+by^k=1\\
  & \text{has a solution $(x, y) \in \zet^2$}\}.\\
\end{align*}
We can now state our main result on random Thue equations.
\begin{thm}
\label{punta_arenas}
Let $k \ge 3$, and assume the truth of the $abc$-conjecture.
Then
\[
  \frac{N_{k, \operatorname{glob}}(H)}{N_{k, \operatorname{loc}}(H)} \rightarrow 0
  \quad (H \rightarrow \infty).
\]
\end{thm}

In particular, assuming the $abc$-conjecture,
for any fixed degree at least three
there are infinitely many Thue equations
violating the Hasse principle, and a `random' Thue
equation of degree at least three
that is locally soluble everywhere has an integer solution
with probability $0$. Theorem \ref{punta_arenas} 
follows immediately from Lemma \ref{chile} and Lemma \ref{sonnig},
whose proofs will be
given in sections \ref{local} and \ref{sec:Thue2}, respectively.
Our strategy roughly follows that laid out in \cite{Bruedern_Dietmann_14},
reversing the roles of linear variables and $k$-th powers when dealing
with equations on average, though the details are simpler here. With
a little bit more work also more general Thue equations of the form
$ax^k+by^k=c$ should be doable, though we refrained from doing so and
concentrated on the special case $c=1$ in order to keep the exposition
simple.

In a similar way one can establish results for homogenized Thue equations,
i.e. Fermat equations. Let
\begin{align*}
  M_{k, \operatorname{loc}}(H) =
  \# & \{a, b, c \in \zet: 0<|a|,|b|,|c| \le H \text{ and }
  ax^k+by^k+cz^k=0\\
  & \text{has non-trivial
  solutions over all local rings $\zet_p$ and over $\er$}\}\\
\end{align*}
and
\begin{align*}
  M_{k, \operatorname{glob}}(H) =
  \# & \{a, b, c \in \zet: 0<|a|,|b|, |c| \le H \text{ and }
  ax^k+by^k+cz^k=0\\
  & \text{has a solution $(x, y, z) \in \zet^3 \backslash
  \{\mathbf{0}\}$}\}.\\
\end{align*}
The following result is a homogeneous analogue of Theorem
\ref{punta_arenas}.
\begin{thm}
\label{thm:bonn}
Let $k \ge 6$, and assume the truth of the $abc$-conjecture.
Then
\[
  \frac{M_{k, \operatorname{glob}}(H)}{M_{k, \operatorname{loc}}(H)} \rightarrow 0
  \quad (H \rightarrow \infty).
\]
\end{thm}
Again, Theorem \ref{thm:bonn} follows immediately from Lemma
\ref{chile2} and Lemma \ref{tesco},
to be proved in sections \ref{local} and
\ref{sec:Thue2}, respectively.

\section{Local considerations}
\label{local}
To get a better understanding of $N_{k, \operatorname{loc}}$ we need the following
two well known results, which we state for the convenience of the reader.
\begin{lemma}
\label{goettingen}
Let $k \in \en$, let $p$ be a rational prime exceeding
$k^2(k+1)^2$,  and let $a_1, a_2, a_3 \in \zet$ be coprime to $p$.
Then the congruence
\[
  a_1 x_1^k + a_2 x_2^k + a_3 \equiv 0 \pmod p
\]
has at least one solution.
\end{lemma}
\begin{proof}
See formula (1.17) in \cite{Mordell}.
\end{proof}

\begin{lemma}
\label{hensel}
Let $f(X_1, \ldots, X_s) \in \zet[X_1, \ldots, X_s]$ and let $p$ be
a rational prime.
Suppose that for some $x_1, \ldots, x_s \in \zet$ and some non-negative
integer $n$ we have
\[
  f(x_1, \ldots, x_s) \equiv 0 \pmod {p^{2n+1}}
\]
and
\[
  p^n || \nabla f(x_1, \ldots, x_s).
\]
Then there exist $y_1, \ldots, y_n \in \zet_p$ such that
\[
  y_i \equiv x_i \pmod {p^{n+1}} \quad (1 \le i \le s)
\]
and $f(y_1, \ldots, y_s)=0$.
\end{lemma}
\begin{proof}
This is a version of Hensel's lemma, see for example page 64 in
\cite{Greenberg}.
\end{proof}
\begin{lemma}
\label{hitze}
Let $p$ be an odd prime and $a \in \zet$ with $(a/p)=1$.
Further, let $k \in \en$ such that $p \equiv -1 \pmod k$. Then the
congruence $x^k \equiv a \pmod p$ has a solution.
\end{lemma}
\begin{proof}
Let $G$ be the multiplicative group of non-zero residue classes modulo
$p$, and let $\varphi : G \rightarrow G$ be the map given by
$\varphi(x)=x^k$ for $x \in G$.
If $k$ is odd, then $p \equiv -1 \pmod k$ implies that $(p-1, k)=1$, so
$\varphi$ is surjective and the conclusion
immediately follows. If $k$ is even, then
$(p-1, k)=2$ by $p \equiv -1 \pmod k$, so $\varphi(G)$ is a subgroup
of $G$ of index $2$. Since $G$ is cyclic, the only such subgroup is
the group of quadratic residues modulo $p$, and as $(a/p)=1$, the
conclusion follows again.
\end{proof}
We are now in a position to derive a lower bound for $N_{k, \operatorname{loc}}$.
\begin{lemma}
\label{chile}
We have
\[
  N_{k, \operatorname{loc}}(H) \gg \left( \frac{H}{\log H} \right)^2.
\]
\end{lemma}
\begin{proof}
For each rational prime $p$, define $\alpha_p$ by $p^{\alpha_p}
\mid \mid k$, and let
\begin{equation}
\label{no_online_checkin}
  m = \prod_{p \le k^2(k+1)^2} p^{2\alpha_p+2}.
\end{equation}
By the Siegel-Walfisz Theorem (see for example Corollary 5.29 in
\cite{Iwaniec-Kowalski}), there are
\[
  \gg_k \left( \frac{H}{\log H} \right)^2
\]
pairs of primes $q, r$ such that
$k^2(k+1)^2<q,r \le H$,  $q \ne r$ and $q \equiv r \equiv -1
\pmod m$. In particular, we then have
$q \equiv r \equiv -1 \pmod k$ and
$q \equiv r \equiv 3 \pmod 4$, so by the law of quadratic
reciprocity, for each such pair $(q, r)$ either
\begin{equation}
\label{terasse}
  \left( \frac{q}{r} \right) = 1 = -\left( \frac{r}{q} \right)
\end{equation}
or $(r/q)=1=-(q/r)$.
By interchanging the roles of $q$ and $r$ if necessary, we may without
loss of generality assume that there are $\gg (H/\log H)^2$ such pairs
$(q, r)$ as above for which the first alternative \eqref{terasse}
holds true. With respect to Lemma \ref{chile} it is then enough to show
that for each such fixed pair $(q, r)$ the Thue equation
\begin{equation}
\label{kopenhagen}
  qx^k-ry^k = 1
\end{equation}
has local solutions everywhere. Since $q>0$ and $r>0$, there are clearly
real solutions, so let us focus on $p$-adic solubility for any given
rational prime $p$.
Let us first discuss the case that $p \le k^2(k+1)^2$. In particular,
$p$ is then coprime to $r$.
Then in order to find a solution of \eqref{kopenhagen} in $\zet_p$,
by Lemma \ref{hensel} it suffices to find a solution of the congruence
\begin{equation}
\label{cong}
  qx^k-ry^k \equiv 1 \pmod {p^{2\alpha_p+1}}
\end{equation}
with $p$ not dividing $y$.
As $r \equiv -1 \pmod m$, by \eqref{no_online_checkin} also
$r \equiv -1 \pmod {p^{2\alpha_p+1}}$, so $x=0, y=1$ is such a solution.
Next, let us assume that $p>k^2(k+1)^2$.
Then $(p,k)=1$, so $\alpha_p=0$.
If $p$ is different from $q$ and $r$, then Lemma \ref{goettingen}
provides a non-singular
solution of \eqref{cong}, which again by Lemma \ref{hensel}
can be lifted to a solution of \eqref{kopenhagen} over $\zet_p$.
Finally, it remains to discuss the two cases $p=q$ and $p=r$. In both
cases, $\alpha_p=0$. For $p=q$, as above we need to find a solution of
the congruence
\begin{equation}
\label{pq}
  -ry^k \equiv 1 \pmod q.
\end{equation}
Now $q \equiv -1 \pmod k$, so by Lemma \ref{hitze} this can be done
providing that $(-r/q)=1$, and the latter condition follows from
$q \equiv 3 \pmod 4$ and \eqref{terasse}. For $p=r$, we need to solve
\[
  qx^k \equiv 1 \pmod r.
\]
Again, $r \equiv -1 \pmod k$, equation \eqref{terasse} and
Lemma \ref{hitze} provide a solution of the
latter congruence. This finishes the proof of Lemma \ref{chile}.
\end{proof}
\begin{lemma}
\label{chile2}
We have
\[
  M_{k, \operatorname{loc}}(H) \gg \left( \frac{H}{\log H} \right)^3.
\]
\end{lemma}
\begin{proof}
The proof is similar to that of Lemma \ref{chile}.
We call a triple $(p_1, p_2, p_3)$ of distinct primes $p_i$
with $p_i \equiv 3 \pmod 4 \; (1 \le i \le 3)$ \emph{good} if there
exists $i \in \{1,2,3\}$ such that
\[
  \left( \frac{p_i}{p_j} \right) =
  \left( \frac{p_i}{p_k} \right),
\]
where $\{i, j, k\}=\{1,2,3\}$. Clearly, for any given quadruple
$(p_1, p_2, p_3, p_4)$ of distinct
primes $p_i$ with $p_i \equiv 3 \pmod 4
\; (1 \le i \le 4)$, we can find three amongst them, say
$p_1, p_2, p_3$, such that $(p_1, p_2, p_3)$ is a good triple.
Now define $m$ by
\eqref{no_online_checkin}.
Then by the Siegel-Walfisz Theorem, and the observation above, we can find
\[
  \gg_k \left( \frac{H}{\log H} \right)^3
\]
triples of distinct primes $q, r, s$ such that
$k^2(k+1)^2<q,r, s \le H$, $q \equiv r \equiv s \equiv -1
\pmod m$
and
\begin{equation}
\label{freitag}
  \left( \frac{s}{q} \right) = \left( \frac{s}{r} \right).
\end{equation}
Note that automatically
$q \equiv r \equiv s \equiv -1 \pmod k$ and
$q \equiv r \equiv s \equiv 3 \pmod 4$.
Now fix any such triple $(q,r,s)$. Using \eqref{freitag},
$q \equiv r \equiv s \equiv 3 \pmod 4$ and the law of quadratic
reciprocity, we find that either
\begin{equation}
\label{alt1}
  \left( \frac{-rs}{q} \right) = \left( \frac{qs}{r} \right) = 1
\end{equation}
or
\begin{equation}
\label{alt2}
  \left( \frac{-rs}{q} \right) = \left( \frac{qs}{r} \right) = -1.
\end{equation}
In the first case, let us consider the equation
\begin{equation}
\label{type_1}
  qx^k-ry^k-sz^k=0.
\end{equation}
There are clearly non-trivial real solutions, and for
$p \le k^2(k+1)^2$ we can follow the argument from the proof of Lemma
\ref{chile} to show that there are non-trivial $p$-adic zeros:
As $q \equiv r \pmod m$, also $q \equiv r \pmod {p^{2\alpha_p+1}}$,
so $(x,y,z)=(1,1,0)$ is a solution of 
\begin{equation}
\label{cong_type_1}
  qx^k-ry^k-sz^k \equiv 0 \pmod {p^{2\alpha_p+1}},
\end{equation}
which by Lemma \ref{hensel} can be lifted to a non-trivial solution of
\eqref{type_1} over $\zet_p$.
If $p>k^2(k+1)^2$, then $\alpha_p=0$.
If in addition $p$ is different from $q,r,s$,
then we can set $z=1$ and use Lemma \ref{goettingen}
to find a non-singular solution of \eqref{cong_type_1}, which again by Lemma
\ref{hensel} lifts to a non-trivial solution of \eqref{type_1} over
$\zet_p$, so it
remains to discuss the case $p \in \{q,r,s\}$. Then $\alpha_p=0$,
so by Lemma \ref{hensel} it suffices to find a non-singular solution of
\[
  qx^k-ry^k-sz^k \equiv 0 \pmod p.
\]
If $k$ is odd, this is easy, since the map $x \mapsto x^k$
is surjective modulo $p$,
as $q \equiv r \equiv s \equiv -1 \pmod k$.
For even $k$, by Lemma \ref{hitze}, it
is enough to find a non-singular solution of
\begin{equation}
\label{flat}
  qx^2-ry^2-sz^2 \equiv 0 \pmod p.
\end{equation}
For $p \in \{q,r\}$ this immediately follows from \eqref{alt1}. For $p=s$,
note that \eqref{freitag} and $q \equiv r \equiv s \equiv 3 \pmod 4$
imply that
\[
  \left( \frac{qr}{s} \right) = 1,
\]
again showing that \eqref{flat} has a non-singular solution. (In fact,
by the Hasse principle for ternary quadratic forms
(see for example Corollary 3 on page 43 of \cite{Serre_GTM}),
as we had already shown non-trivial local solubility of $qx^2-ry^2-sz^2=0$
over $\er$ and all local fields except possibly $\qu_s$, the existence of
a non-trivial solution over $\qu_s$ would have followed automatically, but
we preferred to show it directly.)
Let us now
briefly discuss the second case \eqref{alt2}. Instead of \eqref{type_1},
we now consider the equation
\[
  qx^k-ry^k+sz^k = 0.
\]
The only slight difference then is the argument for $p \in \{q,r,s\}$
and even $k$.
Again, we need to make sure that
\[
  qx^2-ry^2+sz^2 \equiv 0 \pmod p
\]
has a non-singular solution, which reduces to checking that
\[
  \left( \frac{rs}{q} \right) = \left( \frac{-qs}{r} \right)
  = \left( \frac{qr}{s} \right) = 1,
\]
and again these properties follow from
$q \equiv r \equiv s \equiv 3 \pmod 4$, \eqref{freitag} and \eqref{alt2}.

\end{proof}
Regarding upper bounds, note that an application of the large sieve gives
\[
  M_{k, \operatorname{loc}}(H) \ll \frac{H^3}{(\log H)^{\Psi(k)}}
\]
where
\[
  \Psi(k) = \frac{3}{\phi(k)}\left(1-\frac{1}{k}\right)
\]
and $\phi$ denotes Euler's totient function
(see Theorem 1.1 in \cite{Browning-Dietmann09}; for composite $k$,
the bound could be improved somewhat). It would be interesting to
decide what is the true order of magnitude for this quantity.
In this direction, for
$k=2$, Hooley \cite{Hooley93} and independently Guo \cite{Guo}
obtained the sharp bound
$M_{2, \operatorname{loc}}(H) \gg H^3/(\log H)^{3/2}$.
\section{The density of soluble Thue equations}
\label{sec:Thue2}

To prove Theorem \ref{punta_arenas}, it remains to bound the quantity
$N_{k,\operatorname{glob}}(H)$ from above,
assuming the truth of the $abc$-conjecture.
To this end, in the case $k=3$, we shall use the results of
\cite{Dietmann-Marmon},
whereas for larger $k$, an elementary argument will suffice. We shall prove the following result.
\begin{lemma}
\label{sonnig}
Assume the truth of the $abc$-conjecture. Then we have
\[
  N_{k, \operatorname{glob}}(H) \ll \begin{cases}
                      H^{47/27+\varepsilon} & \text{for } k = 3,\\
                      H^{1+ \epsi} & \text{for } k \geq 4.
                     \end{cases}
\]
\end{lemma}
The proof of Lemma \ref{sonnig} begins with the following observation: let $a, b, x, y \in \zet$ where $0<\max\{|a|, |b|\} \le H$.
Suppose that
\[
  ax^k + by^k = 1.
\]
If the $abc$-conjecture holds true, then
\begin{equation}
\label{eq:abc1}
  \max\{|ax^k|, |by^k|\} \ll
  \left( \prod_{p \, | \, abx^ky^k} p \right)^{1+\varepsilon}
  \ll |abxy|^{1+\varepsilon}.
\end{equation}
By symmetry,
without loss of generality, we can assume that $|y| \ge |x|$. Then
\[
  |y^k| \ll H^{1+\varepsilon} |x|^{1+\varepsilon} |y|^{1+\varepsilon}
  \ll H^{1+\varepsilon} |y|^{2+\varepsilon},
\]
so
\begin{equation}
\label{eq:abc2}
|y| \ll H^{1/(k-2)+\varepsilon}, \quad \text{whence also} \quad |x| \ll H^{1/(k-2)+\varepsilon}.
\end{equation}

Next, let us recall the result from \cite{Dietmann-Marmon} that we will use. Let $N(X,Y,Z)$ be the number of quadruples $(a,b,x,y)\in\NN^4$ satisfying
\begin{equation}
\label{eq:dioph}
\begin{gathered}
ax^k-by^k=1, \\
X< x \leq 2X, \quad Y< y\leq 2Y \quad \text{and} \quad Z< by^k\leq 2Z.
\end{gathered}
\end{equation}
The following proposition summarizes the main technical result in \cite{Dietmann-Marmon}. Its proof relies on a recent version of the approximate
determinant
method by Heath-Brown (see \cite{Heath-Brown12}).
\begin{prop}
\label{prop:determinant}
Suppose that $X \leq Y \ll Z^{1/k} \ll XY$. Let $M$ be a natural number satisfying
\begin{equation}
\label{eq:logM}
\log Z \geq \log M \geq \max\left\{ \frac92(1 + \delta)\frac{\log(ZX^{-k}) \log Y}{\log Z}, \log Y \right\}
\end{equation}
for a given $\delta >0$. Then we have the estimate
\begin{equation}
\label{eq:determinant}
N(X,Y,Z) \ll_{\delta,\epsi} Z^{\epsi} (XM^{1/2} + Y). 
\end{equation}
If instead $X \geq Y$, then the same holds with the roles of $X$ and $Y$ interchanged in \eqref{eq:logM} and \eqref{eq:determinant}.
\end{prop}

Proposition \ref{prop:determinant} is valid for all $k \geq 3$, but we shall in fact only need it for $k=3$. Indeed, if $k \geq 4$, then by \eqref{eq:abc2} we immediately have
\begin{align*}
N_{k, \operatorname{glob}}(H) &\leq \#\{ a,b,x,y \in \ZZ: 0< |a|, |b| \le H, |x|,|y| \ll H^{1/2 + \epsi}, ax^k+by^k = 1\} \\
& = \sum_{\substack{|x|,|y| \ll H^{1/2 + \epsi}\\ (x,y) = 1}} \#\{a,b \in \ZZ: 0< |a|, |b| \le H, ax^k+by^k = 1\} \\
& \ll \sum_{|x|,|y| \ll H^{1/2 + \epsi}} \left(1+ \frac{H}{\max\{|x|^k,|y|^k\}}\right) \ll H^{1 + \epsi},
\end{align*}
as asserted in Lemma \ref{sonnig}.

In view of Proposition \ref{prop:determinant}, it will now be more convenient to study the quantity
\begin{align*}
N^+_{k, \operatorname{glob}}(H)  = \# & \{a, b \in \NN: a,b \le H \text{ and }
  ax^k-by^k=1\\
  & \text{has a solution $(x, y) \in \NN^2$}\}. 
\end{align*}
We certainly have $N_{k, \operatorname{glob}}(H) \ll N^+_{k, \operatorname{glob}}(H)$.
From now on, we let $k=3$. Again, by \eqref{eq:abc2} we have
\begin{align*}
N^+_{3, \operatorname{glob}}(H) &\leq \#\{ a,b,x,y \in \NN: a, b \le H,\ x,y \ll H^{1 + \epsi},\ ax^3-by^3 = 1\} \\
& = \sum_{x,y \ll H^{1 + \epsi}} \#\{a,b \in \NN: a, b \le H,\ ax^3-by^3 = 1\}. 
\end{align*}
For a parameter $Q$ to be specified at a later stage, we shall estimate separately the contributions to $N^+_{3, \operatorname{glob}}(H)$ from terms with $xy \leq Q$ and terms with $xy > Q$, respectively. The contribution from the first range is
\begin{equation}
\label{eq:first_range}
\ll \sum_{xy \leq Q} \left(1+ \frac{H}{\max\{x^3,y^3\}}\right) \ll Q \log Q + H.
\end{equation}

For the remaining range, we shall use Proposition \ref{prop:determinant}. Indeed, partitioning the ranges for $x$, $y$ and $by^k$ into dyadic intervals, we obtain
\[                      
\sum_{x,y \ll H^{1 + \epsi}} \#\{a,b \in \NN: a, b \le H,\ ax^3-by^3 = 1\} \ll H^\epsi \max_{X,Y,Z} N(X,Y,Z),
\]                       
where the maximum is taken over $X,Y,Z$ satisfying the conditions
\[
Z \ll H^{4+\epsi}, \quad  X, Y \ll H^{1+\epsi}, \quad  XY \gg Q.
\]
Thus, let $X,Y,Z$ as above be fixed such that $N(X,Y,Z)$ is maximal. Without loss of generality, we may assume that $X \leq Y \ll Z^{1/3}$, so if we require that $Q \gg H^{4/3 + \epsi}$, then Proposition \ref{prop:determinant} is applicable. Let us write $Z = H^\tau$, where $\tau \leq 4 + \epsi$. If $Q = H^\gamma$, then we may write $X \approx Z^\alpha$ and $Y \approx Z^\beta$, where
\begin{equation}
\label{eq:restrictions}
\alpha \leq \beta \leq \min\left\{1/3,1/\tau\right\}, \quad \alpha + \beta \geq \gamma/\tau. 
\end{equation}
In view of \eqref{eq:logM}, we choose $\delta$, depending on $\epsi$, such that $$\frac{9}{2}\delta(1-3\alpha)\beta \leq \epsi,$$ and we take $M \in \NN$ to satisfy 
\begin{equation}
\label{eq:chooseM}
\max\left\{Z^{\frac{9}{2} (1+\delta) (1-3\alpha)\beta},Z^{\beta}\right\} \leq M \ll \max\left\{Z^{\frac{9}{2}(1+\delta)(1-3\alpha)\beta}, Z^{\beta}\right\}.
\end{equation}
Provided that $M \leq Z$, the estimate \eqref{eq:determinant} then yields
\begin{align*}
N(X,Y,Z) &\ll Z^\epsi \left( Z^{\alpha + \frac{9}{4}(1-3\alpha)\beta} + Z^{\alpha + \frac12 \beta} + Z^\beta\right) \\
&\ll H^\epsi \left( H^{u + \frac{9}{4}\left(1-\frac{3}{4}u\right)v} + H^{u + \frac12 v} + H^v\right), 
\end{align*}
where we have put $u = \tau\alpha$ and $v = \tau\beta$ and used that $\tau \leq 4+\epsi$. For $u$ and $v$, we have the restrictions
\begin{equation}
\label{eq:region}
u \leq v \leq 1, \quad u+v \geq \gamma.  
\end{equation}

The two terms $H^{u + \frac12 v}$ and $H^v$ now obviously give negligible contributions to $N^+_{3, \operatorname{glob}}(H)$. Moreover, for $u,v$ satisfying the inequalities \eqref{eq:region}, the function
\[
\Psi(u,v) = u + \frac94\left(1-\frac34 u\right)v
\]
appearing in the exponent of the remaining term satisfies
\[
\Psi(u,v) \leq \Psi(u,1) =  \frac{9}{4} - \frac{11}{16}u \leq \frac{9}{4} - \frac{11}{16}(\gamma - 1) = \frac{47 - 11\gamma}{16}.
\]
In view of the estimate \eqref{eq:first_range}, we optimize by equating the rightmost expression to $\gamma$, taking $\gamma = 47/27$. To establish the estimate $N_{3, \operatorname{glob}}(H) \ll H^{47/27+\epsi}$, it remains only to justify the assumption $M \leq Z$. To this end, we analyze the quantity
\[
\Phi(\alpha,\beta) =  \frac{9}{2}(1-3\alpha)\beta
\]
appearing in \eqref{eq:chooseM}. Note that the assumptions \eqref{eq:restrictions} imply $\alpha \geq (\gamma-1)/\tau = 20/27\tau$, so that
\[                                                                                                                                                                                                                                                                                                                                                                                        
\Phi(\alpha,\beta) \leq \frac{9}{2\tau}(1-3 \alpha) \leq \frac{9}{2\tau}\left(1-\frac{20}{9\tau}\right) = g(1/\tau),                                                                                                                                                                                                                                                                                                                                                                                      \]
say, where $g(t) = \frac92t(1-\frac{20}{9}t)$. As the quadratic function $g$ is decreasing for $t \geq 9/40$, we have
\[
\Phi(\alpha,\beta) \leq g(1/(4+\epsi)) \leq g(9/40) = 81/160 < 1,
\]
so we may certainly ensure that $M \leq Z$ be choosing $\delta$ small enough. This finishes the proof of Lemma \ref{sonnig}.

\section{The density of soluble Fermat equations}
As in section \ref{sec:Thue2}, to prove Theorem \ref{thm:bonn},
it remains to establish an upper bound for $M_{k, \operatorname{glob}}(H)$ as given
in the following result.
\begin{lemma}
\label{tesco}
Let $k \ge 6$, and assume the truth of the $abc$-conjecture. Then
\[
  M_{k, \operatorname{glob}}(H) \ll H^{2+\varepsilon}.
\]
\end{lemma}
For technical reasons, it is easier first to deal with the quantity
\begin{align*}
  M_{k, \operatorname{glob}, \operatorname{prim}}(H)  =
  \# & \{a, b, c \in \zet: (a,b,c)=1, 0<|a|,|b|, |c| \le H\\
  & \text{ and }
  ax^k+by^k+cz^k=0
  \text{ has a solution $(x, y, z) \in \zet^3 \backslash
  \{\mathbf{0}\}$}\}
\end{align*}
focusing on Fermat equations $ax^k+by^k+cz^k=0$ with primitive
coefficient vector $(a,b,c)$. Under the assumptions of Lemma \ref{tesco},
we will show that
\begin{equation}
\label{26grad}
  M_{k, \operatorname{glob}, \operatorname{prim}}(H) \ll H^{2+\varepsilon}.
\end{equation}
Now the equation $ax^k+by^k+cz^k=0$ has a non-trivial integer solution
$(x,y,z)$ if and only if the equation
$\frac{a}{\gamma}x^k+\frac{b}{\gamma}y^k+\frac{c}{\gamma}z^k=0$
has one, where $\gamma=(a,b,c)$. Therefore, 
it is easy to deduce Lemma \ref{tesco} from \eqref{26grad} via
\[
  M_{k, \operatorname{glob}}(H) \le \sum_{\gamma \le H}
  M_{k, \operatorname{glob}, \operatorname{prim}}(H/\gamma)
  \ll \sum_{\gamma \le H} \left( \frac{H}{\gamma} \right)^{2+\varepsilon}
  \ll H^{2+\varepsilon}.
\]
Thus it remains to prove \eqref{26grad}, so
suppose that $ax^k+by^k+cz^k=0$ has a solution $(x,y,z) \in \zet^3
\backslash\{\mathbf{0}\}$. Since the equation is homogeneous,
we can without loss of generality assume that $(x,y,z)=1$.
Put $u=ax^k$, $v=by^k$, $w=cz^k$, and let $\gamma$ be the greatest
common divisor of $u, v, w$. Now if $p^g \mid \mid \gamma$ for some
prime power $p^g$,
then $p^g$ must divide one of $a, b, c$, because of $(x,y,z)=1$.
Therefore, $\gamma$ divides $abc$, so
\[
  \prod_{p \mid \frac{u}{\gamma} \cdot \frac{v}{\gamma} \cdot
  \frac{w} {\gamma}} p \le
  \prod_{p \mid \frac{abc}{\gamma}} p \cdot
  \prod_{p \mid x^k y^k z^k} p \le
  \frac{|abcxyz|}{\gamma}.
\]
Let us now without loss of generality assume that $|x| \ge |y|$ and
$|x| \ge |z|$. As
$\frac{u}{\gamma}+\frac{v}{\gamma}+\frac{w}{\gamma}=0$,
by the $abc$-conjecture we obtain
\[
  \frac{|u|}{\gamma} \ll
  \left( \frac{|abcxyz|}{\gamma} \right)^{1+\varepsilon},
\]
so
\[
  \max\{|x|, |y|, |z|\} \ll H^{2/(k-3)+\varepsilon}.
\]
Consequently,
\begin{align*}
\label{bonn}
  M_{k, \operatorname{glob}, \operatorname{prim}}(H) \le
  \# & \{a, b, c, x, y, z \in \zet: (a,b,c)=(x,y,z)=1, \nonumber\\
  & |a|,|b|,|c| \le H, |x|, |y|, |z|
  \ll H^{2/(k-3)+\varepsilon}, \nonumber \\
  & \text{ and } ax^k+by^k+cz^k=0\}.
\end{align*}
Now for fixed $x, y, z$ with $(x, y, z)=1$, the
integer solutions $(a,b,c)$ of the equation $ax^k+by^k+cz^k=0$
lie on a two-dimensional lattice $\Gamma$ of determinant
\begin{equation}
\label{nuesse}
  \Delta_{x,y,z} \gg \max\{|x|^k, |y|^k, |z|^k\}
\end{equation}
(see Lemma 4.4 in \cite{Browning09}), and by Lemma 4.5 in
\cite{Browning09}, the number of such primitive solutions $(a,b,c)$
with $|a|, |b|, |c| \le H$ is at most of order of magnitude
\begin{equation}
\label{tim}
  1 + \frac{H^2}{\Delta_{x,y,z}}.
\end{equation}
Hence the contribution to $M_{k, \operatorname{glob}, \operatorname{prim}}(H)$
coming from those $x, y, z$ giving $\Delta_{x,y,z} \ge H^2$ is at most
the order of magnitude $O(H^{6/(k-3)+\varepsilon})$ of all permissible
$(x,y,z)$ stemming from the bound $|x|, |y|, |z| \le
H^{2/(k-3)+\varepsilon}$. Since $k \ge 6$, this is compatible with
\eqref{26grad}. Let us now bound the contribution from smaller
$\Delta_{x, y, z}$. To this end, fix $A \in [1, H^2]$. By \eqref{nuesse},
the number of $x, y, z \in \zet$ such that $A \le \Delta_{x,y,z} \le 2A$
is at most $O(A^{3/k})$, and for such fixed $x, y, z$, by \eqref{tim},
there are at most $O(H^2/A)$ corresponding $(a,b,c)$. The total
contribution from $A \le \Delta \le 2A$ is therefore $O(H^2 A^{3/k-1})$.
A dyadic summation over the range of $A$, keeping in mind that $k \ge 6$,
therefore again
gives the bound $O(H^{2+\varepsilon})$ as claimed in \eqref{26grad}.
This finishes the proof of Lemma \ref{tesco}.

\bibliographystyle{plain}
\bibliography{thue_fermat_bib2}{}

\end{document}